\documentclass[12pt]{amsart}
\usepackage{amsfonts,amssymb,amsmath,amsthm}
\usepackage{enumerate}
\usepackage{bbm}
\usepackage{times}
\numberwithin{equation}{section}

\newtheorem{theor}{Theorem}[section]

\newtheorem{prop}[theor]{Proposition}

\newcounter{other}            
\newtheorem{otherth}[other]{Theorem}              

\def\B{\mathcal{B}}
\def\D{\mathbb{D}}
\def\C{\mathbb{C}}
\def \T {\mathbb{T}}
\def \f {\frac}
\def \ind {\int_\D}
\def \H {\mathcal{H}}
\def \qp {\mathcal{Q}_p}
\def \Q {\mathcal Q}

\begin{document}

\title
{A problem for Hankel measures on  Hardy space}
\author{Guanlong Bao and Fangqin Ye}
\address{Guanlong Bao\\
Department of Mathematics\\
    Shantou University\\
    Shantou, Guangdong 515063, China}
\email{glbao@stu.edu.cn}

\address{Fangqin Ye\\
Business School\\
    Shantou University\\
    Shantou, Guangdong 515063, China}
\email{fqye@stu.edu.cn}

\thanks{The work  was   supported  by NNSF of China (No. 11720101003 and No. 11571217), Department of Education of Guangdong Province (No. 2017KQNCX078) and
STU Scientific Research Foundation for Talents (No. NTF17020 and No. STF17005).}

\begin{abstract}
In this note, we investigate  a condition related to the characterization of Hankel measures on Hardy space. We
address a problem    mentioned by J. Xiao in   2000.
\end{abstract}

\maketitle

\section{Introduction}

Let $\D$ be the open unit disk in the complex plan $\C$ and let  $H(\D)$ be the space of analytic functions in $\D$.
Recall that for $0<p<\infty$ the Hardy space  $H^p$  consists  of  functions  $f\in H(\D)$ such that
$$
\|f\|_{H^p}=\sup_{0<r<1} \left(\frac{1}{2\pi}\int_0^{2\pi}|f(re^{i\theta})|^p d\theta \right)^{1/p}<\infty.
$$
It is well known that every function $f\in H^p$ has non-tangential  limits $f(\zeta)$ for almost every $\zeta$ on the unit circle $\T$.
See  \cite{D} for the theory of Hardy spaces.

Let   $BMOA$ be the space of analytic functions of bounded mean oscillation. It is well known (cf. \cite{B, G}) that
the space $BMOA$ can be defined as the set of functions $f\in H(\D)$ satisfying
$$
\|f\|_{BMOA}^2=\sup_{a\in \D} \ind |f'(z)|^2 (1-|\sigma_a(z)|^2)dA(z)<\infty,
$$
where
$$
dA(z)=\frac{1}{\pi}dxdy=\frac{r}{\pi} drd\theta, \ \ z=x+iy=re^{i\theta},
$$
and
$$
\sigma_a(z)=\frac{a-z}{1-\overline{a}z}
$$
is the  M\"obius transformation of   $\D$ interchanging $a$ and $0$. The Fefferman-Stein duality theorem tells us that the dual space of $H^1$ is $BMOA$.  Also, $BMOA$ is a
proper subset of the Bloch space $\B$ consisting of functions
$f\in H(\D)$ for which
$$
\|f\|_{\B}=\sup_{z\in \D} (1-|z|^2)|f'(z)|<\infty.
$$

Given an arc $I$ of $\T$ with arclength $|I|$, the
 Carleson box  $S(I)$ is
$$
S(I)=\left\{r\zeta \in \D: 1-\f{|I|}{2\pi}<r<1, \ \zeta\in I\right\}.
$$
A complex measure $\mu$ on $\D$ is called a {\it Carleson measure} if
$$
\sup_{I\subseteq \T} \f{|\mu|(S(I))}{|I|}<\infty.
$$
It is well known (cf. \cite{D} ) that $\mu$ is a  Carleson measure if and only if there exists a positive constant $C$ such that
$$
\ind |f(z)|^2 d|\mu|(z)\leq C \|f\|^2_{H^2}
$$
for all $f\in H^2$.

Following J. Xiao \cite{X1},  a complex measure $\mu$ on  $\mathbb{D}$ is said to be a {\it Hankel measure } if there exists a positive constant $C$ such that
\begin{equation} \label{Hankel measure}
\left|\int_\D f^2(z) d\mu(z) \right|\leq C \|f\|^2_{H^2}
\end{equation}
for all $f$ in $H^2$. It is clear  that if $\mu$ is a Hankel measure, then $|\mu(\D)|<\infty$. The name of Hankel measure has its root in  the study of Hankel matrices (see also  \cite{JP, W}).
Clearly, a Carleson measure must be a Hankel measure.
 Answering a  problem posed by J. Peetre, J. Xiao \cite{X1} gave a series of complete characterizations of Hankel measures  involving Carleson measures, the duality between $H^1$ and $BMOA$, and Hankel operators.

We refer to   \cite[p. 11]{ARSW} and \cite{X2} for the study of complex measure  inequalities  similar to (\ref{Hankel measure}) in the setting of  the classical  Dirichlet space and weighted Bergman spaces, respectively.

From  \cite{X1},  a necessary condition for a complex measure $\mu$ on $\D$ to be a Hankel measure is
\begin{equation} \label{a condition}
\sup_{w\in \D} \left|\ind \frac{1-|w|^2}{(1-\overline{w}z)^2}d\mu(z)\right|<\infty.
\end{equation}
By this observation, J. Xiao \cite[p. 139]{X1}  mentioned  the following problem: Is condition (\ref{a condition}) sufficient for $\mu$ to be a Hankel measure too?
In this note, we show that in general the answer to this problem  is negative.   We give  a complex measure $\mu$ on $\D$ satisfying   condition (\ref{a condition}) but $\mu$ is not a Hankel measure.
A positive Borel measure $\mu$ on [0, 1) can be seen as a Borel measure on $\D$ by identifying it
with the measure $\tilde{\mu}$ defined by
$$
\tilde{\mu}(E)=\mu(E \cap [0, 1)),
$$
for any Borel subset $E$ of $\D$.
If $\mu$ is a positive Borel  measure on [0, 1),  we obtain that  condition (\ref{a condition}) holds if and only if  $\mu$ is  a Hankel measure, which also implies   some new information of a Hankel
matrix acting on  Hardy spaces and Dirichlet type spaces.

\section{Condition (\ref{a condition}) and Hankel measures}

The section is devoted to consider  whether   condition (\ref{a condition}) is  an equivalent description of Hankel measure.
For a complex measure $\mu$ on $\D$, the function $P_\mu$ is given by
$$
P_{\mu}(w)=\ind \f{1}{1-w\overline{z}}d\mu(z), \ \ w\in \D.
$$
We will show that the  problem considered in this note  is equivalent to investigate  those functions $P_{\overline{\mu}} \in \B$  must actually satisfying $P_{\overline{\mu}} \in BMOA$.

There exist several  complete descriptions of Hankel measures in \cite{X1}. Now we cite only  some of them as follows.

\begin{otherth}\label{Xiao's theorem}
Let $\mu$ be a complex measure on $\D$. Then the following conditions are equivalent.
\begin{enumerate}
  \item [(1)] $\mu$ is a Hankel measure.
  \item [(2)] There exists a positive constant $C$ such that
   $$
  \left|\int_\D f(z) d\mu(z) \right|\leq C \|f\|_{H^1},
  $$
  for all $f\in H^1$.
  \item [(3)] $P_{\overline{\mu}}$ is in $BMOA$.
  \item [(4)]
  $$
  \sup_{I\subseteq \T} \frac{1}{|I|} \int_{S(I)} \left|\ind\f{\overline{z}d\overline{\mu}(z)}{(1-w\overline{z})^2}\right|^2 (1-|w|^2)dA(w)<\infty.
  $$
\end{enumerate}
\end{otherth}

The following theorem characterizes lacunary series in the Bloch space $\B$ and $BMOA$ (see  \cite{G} and \cite{P}).

\begin{otherth}\label{gap series}
Let $f\in H(\D)$ with the power series expansion   $f(z)=\sum_{k=1}^\infty a_k z^{n_k}$  and suppose there exists
$\lambda >1$ such that $n_{k+1}\geq \lambda n_k$ for all $k$. Then the following assertions hold.
\begin{enumerate}
  \item [(1)] $f\in \B$ if and only if the sequence $\{a_k\}$ is bounded.
  \item [(2)] $f\in BMOA$ if and only if
  $$
  \sum_{k=1}^\infty |a_k|^2<\infty.
  $$
\end{enumerate}
\end{otherth}

 In 1995, R. Aulaskari,  J. Xiao and  R. Zhao \cite{AXZ} introduced  $\qp$ spaces which  attracted a lot of attention in recent years.
 For
$0<p<\infty$,  the space $\qp$ consists of functions $f\in H(\D)$ with
$$
\|f\|_{\qp}^2 =\sup_{a\in\D}\,\int_{\D}
|f'(z)|^{2}\left(1-|\sigma_a(z)|^2\right)^pdA(z)<\infty.
$$
 Clearly, $\Q_1=BMOA$.
By \cite{AL}, we know that for  $1<p<\infty$, all $\qp$ spaces are the same and equal to the Bloch space $\B$. See J. Xiao's monographs  \cite{X3, X4} for more  results  of $\qp$ spaces.

The following result can be founded in \cite[p. 29]{X3}.

\begin{otherth}\label{nonnegative nonincreasing}
Suppose $0<p<\infty$.  Let $f(z)=\sum_{n=0}^\infty a_n z^n$ with $\{a_n\}$ being a decreasing sequence of nonnegative numbers. Then
$f\in \qp$ if and only if $a_n=O(\f{1}{n})$.
\end{otherth}

We first show that the problem considered  in this note is associated with  certain  self-improving property of functions $P_{\overline{\mu}}$.
\begin{theor}\label{th}
Let $\mu$ be a complex measure on $\D$. Then the following two statements  are equivalent.
\begin{enumerate}
  \item [(1)] If $\mu$ satisfies condition (\ref{a condition}), then $\mu$ is  a Hankel measure.
  \item [(2)] If $P_{\overline{\mu}} \in \B$, then $P_{\overline{\mu}} \in BMOA$.
\end{enumerate}
\end{theor}
\begin{proof}
For $w\in \D$, one gets
\begin{eqnarray*}
wP_{\overline{\mu}}(w)&=&\ind \f{w}{1-w\overline{z}}d\overline{\mu}(z)\\
&=& \sum_{n=0}^\infty \left(\ind \overline{z}^n d\overline{\mu}(z)\right) w^{n+1}.
\end{eqnarray*}
Then
\begin{eqnarray*}
&~& \sup_{w\in \D} (1-|w|^2) |\left(wP_{\overline{\mu}}(w)\right)'|\\
&=&\sup_{w\in \D} (1-|w|^2) \left|\sum_{n=0}^\infty (n+1) \left(\ind \overline{z}^n d\overline{\mu}(z)\right) w^{n}\right| \\
&=& \sup_{w\in \D} \left|\ind \frac{1-|w|^2}{(1-w\overline{z})^2} d\overline{\mu} (z)\right|\\
&=& \sup_{w\in \D} \left|\ind \frac{1-|w|^2}{(1-\overline{w}z)^2} d\mu (z)\right|.
\end{eqnarray*}
Consequently, the function $ w\rightarrow wP_{\overline{\mu}}(w)$ belongs to $\B$ if and only if condition (\ref{a condition}) holds for $\mu$.

By the growth estimates of functions in $\B$ (see \cite[p. 6]{X3} for example), one gets that there exists a positive $C$
such that
$$
|f(z)-f(0)|\leq C \|f\|_\B \log\f{1}{1-|z|}
$$
for all $f\in \B$.  Bearing  in mind this estimate, we see that   $P_{\overline{\mu}} \in \B$ if and only if the function $ w\rightarrow wP_{\overline{\mu}}(w)$ belongs to $\B$.

Thus we proved that $P_{\overline{\mu}} \in \B$ if and only if $\mu$ satisfies condition (\ref{a condition}). By Theorem \ref{Xiao's theorem},
$\mu$ is a Hankel measure if and  only if $P_{\overline{\mu}} \in BMOA$. These imply the desired result.
\end{proof}

The following result shows   that in general the answer to J. Xiao's problem is negative.
\begin{prop}\label{negative}
There exists a complex measure $\mu$ on $\D$ such that condition (\ref{a condition}) holds but $\mu$ is not a Hankel measure.
\end{prop}
\begin{proof}
Set $d\overline{\mu}(z)=  f(z)dA(z)$ for $z\in \D$, where $f(z)=1+\sum_{k=1}^\infty (1+2^{k})z^{2^k}$. Then $\overline{\mu}(\D)=1$.  By the
Parseval formula, for any positive  integer $n$,
\begin{eqnarray*}
\ind \overline{z}^n d\overline{\mu}(z)&=&\frac{1}{\pi}\int_0^1 r^{n+1}dr \int_0^{2\pi} e^{-in\theta} [1+\sum_{k=1}^\infty (1+2^{k})r^{2^k}e^{i 2^k \theta}]d\theta\\
&=&
\begin{cases}
& \enspace  0, \ \ \text{if} \ \  n \not\in \{2^k: k=1, 2, 3, \cdots\},     \\
& \enspace  1, \ \ \text{if} \ \ n \in \{2^k: k=1, 2, 3, \cdots\}.
 \end{cases}
\end{eqnarray*}
Hence, for any $w\in \D$,
\begin{eqnarray*}
P_{\overline{\mu}}(w)&=&\sum_{n=0}^\infty \left(\ind \overline{z}^n d\overline{\mu}(z)\right) w^{n}\\
&=& 1+  \sum_{k=1}^\infty w^{2^k}.
\end{eqnarray*}
By Theorem \ref{gap series}, $P_{\overline{\mu}} \in \B$ but $P_{\overline{\mu}} \not\in BMOA$. Via Theorem \ref{Xiao's theorem} and the proof of Theorem \ref{th}, we see that
the measure $\mu$ satisfies condition (\ref{a condition}), but $\mu$ is not a Hankel measure.
\end{proof}

For a positive Borel measure $\mu$  on [0, 1) and a nonnegative integer
$n$,  denote by  $\mu[n]$  the moment of order $n$ of $\mu$, that is,
$$
\mu[n]=\int_0^1 t^n d\mu(t).
$$
The following result shows that there exist measures $\mu$ such that if $P_{\overline{\mu}} \in \B$, then $P_{\overline{\mu}} \in BMOA$.  It follows from
Theorem \ref{th} that in this case $\mu$ is a Hankel measure if and only if $\mu$ satisfies condition (\ref{a condition}).

\begin{theor}\label{positive}
Let $\mu$ be a positive Borel measure  on [0, 1).  Then the following conditions are equivalent.
\begin{enumerate}
  \item [(1)] $\mu$ is a Hankel measure.
  \item [(2)]  $\mu$ satisfies condition (\ref{a condition}), that is,
  $$
  \sup_{w\in \D} \left|\int_0^1 \frac{1-|w|^2}{(1-\overline{w}t)^2}d\mu(t)\right|<\infty.
  $$
  \item [(3)] $P_{\mu} \in \mathcal{Q}_p$ for some $0<p<\infty$.
  \item [(4)] $P_{\mu} \in \mathcal{Q}_p$ for all  $0<p<\infty$.
  \item [(5)]  $\mu[n]=O(\frac{1}{n})$.
\end{enumerate}
\end{theor}
\begin{proof}  If $\mu$ is a positive Borel measure  on [0, 1), then for $w\in \D$,
$$
P_{\overline{\mu}}(w)=P_\mu(w)=\sum_{n=0}^\infty \left(\int_0^1 t^n d\mu(t)\right) w^{n}=\sum_{n=0}^\infty \mu[n] w^n.
$$
Clearly, $\left\{ \mu[n] \right\}$ is a decreasing sequence with nonnegative numbers. By  Theorem \ref{nonnegative nonincreasing}, we
see that (3)$\Leftrightarrow $(4)$\Leftrightarrow$ (5). By the proof of Theorem \ref{th}, $P_{\overline{\mu}} \in \B$ if and only if
$\mu$ satisfies condition (\ref{a condition}).  From Theorem \ref{Xiao's theorem}, $P_{\overline{\mu}} \in BMOA$ if and only if
$\mu$ is Hankel measure. Bear in mind that $\B=\Q_2$ and $BMOA=\Q_1$.  Thus both (1) and (2) are equivalent to (5). The proof is complete.
\end{proof}

Let $\mu$ be a positive Borel measure  on [0, 1). Denote by $\H_\mu$ the Hankel matrix
$(\mu[n+k])_{n, k\geq 0}$. This matrix induces formally an  operator, denoted also by $\H_\mu$, on $H(\D)$ in
the following sense. For $f(z)=\sum_{n=0}^\infty a_n z^n\in H(\D)$, by multiplication of the matrix with the sequence
of Taylor coefficients of the function $f$,
$$
\{a_n\}_{n\geq 0} \mapsto \left\{\sum_{k=0}^\infty \mu[n+k] a_k\right\}_{n\geq 0},
$$
define
$$
\H_\mu (f)(z)=\sum_{n=0}^\infty \left(\sum_{k=0}^\infty \mu[n+k] a_k\right)z^n.
$$
If $\mu$ is the Lebesgue measure on [0, 1), the matrix $\H_\mu$ reduces to the classical Hilbert matrix
$\H=((n+k+1)^{-1})_{n, k\geq 0}$, which induces the classical Hilbert operator $\H$.
The Hankel matrix $\H_\mu$ acting on $H^2$ was studied in \cite{Pow, W}. For all Hardy spaces $H^p$, the operator $\H_\mu$ was investigated in
\cite{C, GM}. The Dirichlet type space $\mathcal{D}_\alpha$, $\alpha \in \mathbb{R}$, consists of functions $f(z)=\sum_{n=0}^\infty a_n z^n\in H(\D)$ for which
$$
\|f\|_{\mathcal{D}_\alpha}^2=\sum_{n=0}^\infty  (n+1)^{1-\alpha}|a_n|^2<\infty.
$$
The Hankel matrix $\H_\mu$ acting on $\mathcal{D}_\alpha$ spaces was considered in \cite{BW, Dia}.  It is worth referring to \cite{GM} for the recent development of
$\H_\mu$ acting on spaces of analytic functions.   The following theorem is from  the related literatures mentioned before.

\begin{otherth}\label{Hankel matrices}
Let $\mu$ be a positive Borel measure  on [0, 1). Suppose $1<p<\infty$ and $0<\alpha<2$.  Then the following conditions are equivalent.
\begin{enumerate}
  \item [(1)] $\mu$ is a Carleson  measure.
  \item [(2)]  $$
  \sup_{w\in \D} \int_0^1 \frac{1-|w|^2}{|1-\overline{w}t|^2}d\mu(t)<\infty.
  $$
  \item [(3)]  $\mu[n]=O(\frac{1}{n})$.
  \item [(4)]  $\mathcal{H}_\mu$ is bounded on $H^p$.
  \item [(5)]  $\mathcal{H}_\mu$ is bounded on $\mathcal{D}_\alpha$.
\end{enumerate}
\end{otherth}

Combining  Theorem \ref{positive} with Theorem \ref{Hankel matrices}, we obtain immediately new characterizations   of a Hankel matrix $\H_\mu$ acting on Hardy spaces
and Dirichlet type spaces in terms of Hankel measures and $\qp$ functions.

\end{document}